\documentclass{amsart}
\author{Nicholas O. Kirby \and Eliot Fried}
\title{Gamma-limit of a model for the elastic energy of an inextensible ribbon}
\usepackage{pdfsync}
\newtheorem{theorem}{Theorem}[section]
\newtheorem{lemma}[theorem]{Lemma}

\newtheorem{corollary}[theorem]{Corollary}
\newtheorem{conjecture}[theorem]{Conjecture}

\newenvironment{definition}[1][Definition]{\begin{trivlist}
\item[\hskip \labelsep {\bfseries #1}]}{\end{trivlist}}

\usepackage{amssymb, amsmath, fullpage,bm,stmaryrd,verbatim,color,amsbsy,graphicx}
\usepackage{epstopdf}

\DeclareMathOperator*{\esssup}{esssup}

\DeclareMathOperator*{\lsc}{sc}
\DeclareMathOperator*{\Glim}{\Gamma-\lim}

\newcommand{\eps}{\varepsilon}

\newcommand{\dA}{\,\mbox{d}A}

\newcommand{\ds}{\,\mbox{d}s}

\newcommand{\RR}{\mathbb{R}}

\newcommand{\NN}{\mathbb{N}}
\newcommand{\bbS}{\mathbb{S}}

\newcommand{\km}{\k_m}
\newcommand{\Xkp}{X^{\km,p}}
\newcommand{\Xikp}{\Xi^{\km,p}}
\newcommand{\Xp}{X^{p}}
\newcommand{\Yk}{Y^{\km}}
\newcommand{\tcF}{\tilde{\cF}}
\newcommand{\wto}{\rightharpoonup}
\newcommand{\ouint}{I}
\newcommand{\cluint}{[0,1]}
\newcommand{\parouint}{(I)}
\newcommand{\cA}{\mathcal{A}}
\newcommand{\cB}{\mathcal{B}}
\newcommand{\cC}{\mathcal{C}}
\newcommand{\cF}{\mathcal{F}}
\newcommand{\cG}{\mathcal{G}}
\newcommand{\cJ}{\mathcal{J}}
\newcommand{\cS}{\mathcal{S}}
\renewcommand{\k}{\kappa}
\newcommand{\bfA}{\mathbf{A}}
\newcommand{\bfa}{\mathbf{a}}
\newcommand{\bfB}{\mathbf{B}}
\newcommand{\bfb}{\mathbf{b}}
\newcommand{\bfc}{\mathbf{c}}
\newcommand{\bfr}{\mathbf{r}}
\newcommand{\bft}{\mathbf{t}}
\newcommand{\bfn}{\mathbf{n}}
\newcommand{\bfw}{\mathbf{w}}
\newcommand{\bfx}{\mathbf{x}}
\newcommand{\bfy}{\mathbf{y}}
\newcommand{\bfz}{\mathbf{z}}
\newcommand{\bfu}{\mathbf{u}}
\newcommand{\pa}{\xi}
\newcommand{\qa}{\zeta}
\newcommand{\bfinf}{\boldsymbol{\infty}}

\newcommand{\dt}{\,{\rm d} t}

\newcommand{\Dp}{\,{\rm d} \pa}
\newcommand{\Dq}{\,{\rm d} \qa}

\usepackage[normalem]{ulem}

\begin{document}

\begin{abstract}
A $\Gamma$-convergence result involving the elastic energy of a narrow inextensible ribbon is established. A non-dimensional form of the elastic energy is reduced to a one-dimensional integral over the centerline of the ribbon with the aspect ratio of the ribbon being a small parameter. That integral is observed to increase monotonically with the aspect ratio.  The $\Gamma$-limit of the family of non-dimensional elastic energies is taken in a Sobolev space of centerlines with non-vanishing curvature.  In that space, it is shown that the $\Gamma$-limit is a functional first proposed by Sadowsky in the context of narrow ribbons that form M\"{o}bius bands.  The results obtained here do not apply to such ribbons, since the centerline of a M\"{o}bius band must have at least one inflection point.
As a first step toward dealing with such inflection points, a result is presented on the lower semicontinuity of the Sadowsky functional with inflection points comprising a set of measure zero within the domain of an arclength parameterization.  



\end{abstract}
\maketitle

\section{Introduction}

An inextensible ribbon is modeled as a two-dimensional surface that is geometrically constrained to be isometric to a rectangle of given length $\ell$ and width $2w$.
The dimensionless parameter $\eps=2w/\ell $ is referred to as the aspect ratio of the ribbon.
Granted that the curvature $\tilde\kappa$ of the centerline $\cC$ of the ribbon is nonvanishing, the geometric constraint yields a parametrization of the ribbon in terms of $\cC$.

To determine the equilibrium shape of an elastic, inextensible ribbon subject to imposed end conditions, it suffices to minimize its net potential energy.  
Here, it is assumed that the elastic energy density $\phi$ of the ribbon is an isotropic, quadratic function of the Weingarten map, and thereby a symmetric, quadratic function of the principle curvatures of the ribbon.
Upon completing the square, $\phi$ admits a representation in terms of the mean and Gaussian curvatures $H$ and $K$ of the ribbon of the form
\begin{equation}\label{genelast}
\phi=\frac{D}{2} (H-H_0)^2+CK,
\end{equation}
where $D$ and $C$ are constant moduli and $H_0$ is the spontaneous mean curvature.
The expression \eqref{genelast} was proposed by Germain \cite{Germain1821}.
The particular version of \eqref{genelast} considered here, in which $H_0$ is taken to be zero, was considered by Poisson \cite{Poisson1812}.

In the limit $\varepsilon\to0$ of vanishing aspect ratio, Sadowsky \cite{Sadowsky1930} argued that the energy of a ribbon forming a M\"{o}bius band should be proportional to 
\begin{equation} \label{Sadowskyfun}
\cF=\int_{\cC} \tilde{\k}^2\left(1+\eta^2\right)^2\Dp,
\end{equation}
where $\eta$ is the ratio $\tilde{\tau}/\tilde{\kappa}$, with $\tilde{\tau}$ the torsion of the centerline $\cC$, and $\pa$ denotes arclength along $\cC$.
The properties of the functional \eqref{Sadowskyfun} were studied in some detail by Wunderlich \cite{Wunderlich1962}. 
Recently, Starostin and van der Heijden \cite{SvdH2007} used the variational bicomplex formalism to investigate the equilibrium equations for the problem associated with minimizing the functional
\begin{equation}\label{Starostinfun}
\cF_{\eps} = \int_{\cC} \tilde{\k}^2 \left(1+\eta^2\right)^2 \frac{1}{\eps\ell \dot{\eta}} \ln \left(\frac{2+\eps\ell \dot{\eta}}{2-\eps\ell\dot{\eta}}\right) \Dp, 
\end{equation}
for $\eps>0$, where a superposed dot indicates differentiation with respect to the arclength parameter $\pa$.
Upon inspection, it is evident that the Sadowsky functional $\cF$ is the pointwise limit of the elastic energy $\cF_{\eps}$ as $\eps \to 0$.
However, the question of whether the Sadowsky functional \eqref{Sadowskyfun} is the proper variational limit (that is, the $\Gamma$-limit) of the elastic energy 
\eqref{Starostinfun} of a ribbon with a given centerline remains unanswered. 
This question is settled herein for curves with nonvanishing curvature that are parametrized by arclength and are elements of certain Sobolev spaces.

As Randrup and R{\o}gen \cite{Randrup1996} remark, the centerline of a nonorientable developable, like a M\"{o}bius band, must have at least one point at which the curvature vanishes.
The problem of establishing the Sadowsky functional as the $\Gamma$-limit for a space of centerlines containing those corresponding to a M\"{o}bius band is left for future work. 
However, a result in this direction is provided.

The first steps in our analysis are identical to those appearing in the papers of Wunderlich \cite{Wunderlich1962} and Starostin and van der Heijden \cite{SvdH2007} and also in the thesis of Yong \cite{Yongthesis}.
These steps deliver an expression for the elastic energy of the ribbon in terms of the shape of its centerline and depending parametrically on the aspect ratio $\eps$ of the ribbon.
In particular, the energy is given by 
\begin{equation}
E=\frac{\eps \ell D}{2} \int_{0}^{\ell} \tilde{\k}^2(\pa) (1+\eta^2(\pa))^2 g\left(\eps\ell \dot{\eta}(\pa)\right) \Dp,
\end{equation}
where $g\colon \RR \to [1,+\infty]$ is defined by
\begin{equation}
g(x) = \begin{cases} 
1, & \text{ if } x=0,\\
\dfrac{1}{x} \ln \left(\dfrac{2+x}{2-x}\right), & \text{ if } |x|<2 \text{ and } x\neq 0,\\
+\infty, & \text{ if } |x|\geq 2.
\end{cases}
\end{equation}
Thereafter, it is shown that, upon taking $\eps$ to zero, the Sadowsky functional 
is not merely the ``pointwise'' limit of the one-parameter family $2E/\eps D$ of functionals 
but is also the $\Gamma$-limit in a naturally chosen space of curves with nonvanishing curvature.
The keys to the argument are the monotonic dependence of the family of energy functionals on $\eps$ and straightfoward applications of the H\"{o}lder inequality.

\section{The energy of an elastic band}
Consider a surface isometric to a rectangle with base $\ell$ and height $2w$. Let
\begin{equation}
\varepsilon=\frac{2w}{\ell}.
\label{eps}
\end{equation} By Gauss's Theorema Egregium, this surface must have Gaussian curvature equal to zero; that is, the surface must be developable.
Struik \cite{Struik1961} proves that any developable surface in $\RR^3$ is a ruled surface.
Given a space curve with length $\ell$, arclength parametrization $\bfr$, Frenet frame $\{\bft,\bfn,\bfb\}$, 
curvature $\tilde{\k}$ with isolated zeroes, and torsion $\tilde{\tau}$, define $\eta$ as the ratio
\begin{equation}\label{etadefa}
\eta=\frac{\tilde{\tau}}{\tilde{\k}}.
\end{equation}
Graustein \cite[\S 52] {Graustein1935} shows that a curve on a developable surface is a geodesic if and only if the surface is the rectifying developable of the curve. 
Graustein \cite[\S 28]{Graustein1935} also shows that the unique rectifying developable of a space curve with centerline having arclength parameterization $\bfr$ and width $2w$ is given by 
\begin{equation}
\bfx (\pa,\qa)=\bfr(\pa)+\qa [ \bfb (\pa) +\eta(\pa) \bft (\pa) ], \quad (\pa,\qa) \in [0,\ell] \times [-w,w].
\end{equation}

Given the centerline $\bfr$, consider the energy of an elastic ribbon $\cS$ of width $2w$ parameterized by $\bfx$.
Wunderlich \cite{Wunderlich1962} shows that the nonvanishing principal curvature $\k_1$ of $\cS$ is given by
\begin{equation}\label{princcurve}
\k_1 = \frac{\tilde{\k} (1+\eta^2)}{|1+\qa\dot{\eta}|},
\end{equation}
where the dot denotes differentiation with respect to arclength.
The bending energy $E$ of $\cS$ therefore takes the form
\begin{equation}\label{explenerg}
E=\frac{D}{2}\int_{\cS} \k_1^2 \dA = \frac{D}{2} \int_{0}^{\ell}\int_{-w}^{w} \frac{\tilde{\k}^2(\pa) (1+\eta^2(\pa))^2}{|1+\qa\dot{\eta}(\pa)|} \Dq\Dp,
\end{equation}
where $D$ is a measure of flexural rigidity.
Evaluating the integral on the far right-hand side of \eqref{explenerg} over the width $2w$ of the ribbon yields
\begin{equation}
E=Dw \int_{0}^{\ell} \tilde{\k}^2(\pa) (1+\eta^2(\pa))^2 g(2w\dot{\eta}(\pa)) \Dp,
\end{equation}
where $g\colon \RR \to [1,+\infty]$ is defined by
\begin{equation}
g(x) = \begin{cases} 
1, & \text{ if } x=0,\\
\dfrac{1}{x} \ln \left(\dfrac{2+x}{2-x}\right), & \text{ if } |x|<2 \text{ and } x\neq 0,\\
+\infty, & \text{ if } |x|\geq 2.
\end{cases}
\end{equation}

\subsection{Nondimensionalization}
To nondimensionalize the problem, introduce a unit speed parametrization $\bfu \colon \cluint \to \RR^3$ of the centerline defined such that $\bfu(s)=~\bfr \left(\ell s\right)/\ell $ for each $s\in \cluint$. Let $\ouint$ denote the open interval $(0,1)$.
For a measurable set $E\subset \cluint$, consider the family of functionals $\cF_{\eps} (\, \cdot\,,E) \colon W^{3,p}(\ouint;\RR^3) \to [0,+\infty]$  defined by
\begin{equation}\label{familycF}
\cF_{\eps} (\bfu, E) = \begin{cases} \displaystyle{\int_{E}  \k^2 \left(1+\eta^2\right)^2 g(\eps \eta') \ds,} & \text{ if } \bfu\in W^{3,p}(\ouint;\RR^3) \text{ such that } \eta \in W^{1,1}\parouint,\\[10pt]
 +\infty, & \text{otherwise}.
\end{cases}
\end{equation}
In \eqref{familycF} and hereafter, a prime indicates differentiation with respect to $s=\pa/\ell $ and the symbols $\k$ and $\tau$ denote the dimensionless counterparts of the curvature $\tilde{\k}$ and torsion $\tilde{\tau}$, given in terms of $\bfu$ by
\begin{equation}
\k= |\bfu''| \qquad\text{and} \qquad \tau=-\frac{\bfu''}{|\bfu''|} \cdot \left[ \frac{\bfu'\times \bfu''}{|\bfu''|}\right]'= \frac{\bfu'\cdot(\bfu''\times \bfu''')}{|\bfu''|^2}.
\end{equation}
By the chain rule, $\k=\ell  \tilde{\k}$ and $\tau=\ell \tilde{\tau}$, whereby $\eta$ as defined in \eqref{etadefa} admits the alternative representation $\eta=\tau/\k$.
It then follows that $\eta$ and $\eta'$ may be expressed in terms of $\bfu$ by
\begin{align}\label{etadef}
	\begin{split}
		\eta={} & \frac{\tau}{\k} = \frac{\bfu' \cdot (\bfu'' \times \bfu''')}{|\bfu''|^3},\\
		\eta'={} & \left [\frac{\bfu' \cdot (\bfu'' \times \bfu''')}{|\bfu''|^3}\right ]'= \frac{\bfu' \cdot (\bfu''\times \bfu'''')}{|\bfu''|^3}-\frac{3 (\bfu''\cdot \bfu''') [\bfu'\cdot (\bfu''\times \bfu''')]}{|\bfu''|^5}.
	\end{split}
\end{align}

\section{Notation}
Since the interval $\ouint=(0,1)$ remains fixed herein, the notation $\| \cdot \|_{p}$ is used to indicate the norms on either of the Lebesgue spaces $L^{p} (\ouint; \RR^3)$ or $L^{p}(\ouint;\RR)$; specifically,
\begin{align}
\begin{split}
\|\bfu\|_{p\phantom{\infty}} ={} &  \left( \int_{\ouint} |\bfu|^p \dt \right)^{1/p} \quad \text{ if } p \text{ is finite,}\\
\|\bfu\|_{\infty\phantom{p}} = {} & \esssup\limits_{\ouint} \{ |\bfu|\}.		
\end{split}
\end{align}

For a function $\bfu: I\to \RR^3$ with $k\ge 1$ continuous weak derivatives $\bfu',\bfu'', \bfu^{3}, \ldots , \bfu^{k}$ and for $t\in I$ let $D^{[k]}\bfu(t)\in (\RR^3)^k$ be defined by
\begin{equation}
D^{[k]}\bfu(t)= (\bfu(t),\bfu'(t),\ldots, \bfu^k(t)).
\end{equation}
Choose $\bfA,\bfB \in \RR^3\times \bbS^2$ (so that, for instance, $\bfA=(\bfa,\bft_0)$ for some $\bfa,\bft_0\in \RR^3$ such that $|\bft_0|=1$). 
Define function spaces $Y$, $\Xp$, $\Yk$, and $\Xkp$ by
\begin{align}\label{fnspace}
  \begin{split}
	Y   := {} & \{ \bfu \in C^{\infty} (\cluint; \RR^3) : D^{[1]}\bfu(0)=\bfA, D^{[1]}\bfu(1)=\bfB, |\bfu'(s)|= 1 \text{ for all } s\in \ouint\},\\
	\Xp  := {} & \text{cl}_{W^{3,p}(\ouint;\RR^3)}(Y) ,\\
	\Yk  := {} & \{ \bfu \in Y : |\bfu''(s)|\geq \k_m \text{ for all } s\in \ouint\},\\
	\Xkp := {} & \text{cl}_{W^{3,p}(\ouint;\RR^3)} (\Yk),
	\end{split}
\end{align}
where for a Banach space $\cB$ and $\cA\subset \cB$ a subset, $\text{cl}_{\cB}(\cA)$ denotes the closure of $\cA$ with respect to the norm of $\cB$.

\section{$\Gamma$-limit}
The goal of the ensuing analysis is to determine a value of $p$ such that the sequence $\{\cF_{\eps}(\cdot, I)\}$ of functionals defined in \eqref{familycF} has $\Gamma$-limit 
\begin{equation}\label{gamlim}
\cF (\bfu,\ouint) = \int_{\ouint} \k^2 (1+\eta^2)^2\ds,
\end{equation}
with respect to weak convergence in $\Xkp$ and with respect to strong convergence in $W^{3,p}(\ouint; \RR^3)$.
The functional $\cF(\cdot,\ouint)$ defined by \eqref{gamlim} is called the Sadowsky functional.
\subsection{Existence of $\Gamma$-limit}
In particular, following Braides \cite{Braides}, given any sequence $\{\eps_j\}$ with $\eps_j>0$ and $\eps_j \to 0$ and any element $\bfu \in X$:
\begin{enumerate}
\item for every sequence $\{\bfu_j\}$ with $\bfu_j\in X$ such that $\bfu_j\to \bfu$ in $X$, $\cF(\bfu,\ouint)$ is bounded above in accord with
\begin{equation}\label{liminfcond}
\cF (\bfu,I) \leq \liminf_{j\to \infty} \cF_{\eps_j} (\bfu_j,I);
\end{equation}
\item there exists a sequence $\{\bfu_j\}$ converging to $\bfu$ such that $\cF(\bfu,\ouint)$ is bounded below in accord with
\begin{equation}\label{limsupcond}
\cF(\bfu,I) \geq \limsup_{j\to \infty} \cF_{\eps_j} (\bfu_j,I).
\end{equation}
\end{enumerate}
For any such sequence $\{\eps_j\}$, it is possible to extract a decreasing subsequence $\{\eps_{j_k}\}$. 
Since the integrand of $\cF_{\eps}(\cdot , \ouint)$ increases with $\eps$ (regardless of the sign of $\eta'$), it can be deduced that (see Remark 1.40 of Braides \cite{Braides}) the $\Gamma$-limit of the sequence of functionals $\cF_{\eps_{j_k}}(\cdot , \ouint)$ exists and is given by
\begin{equation}\label{glima}
\Glim_{k \to \infty} \cF_{\eps_{j_k}}(\cdot , \ouint)=\lsc \left(\inf_{k\in \NN} \cF_{\eps_{j_k}}(\cdot , \ouint)\right)=\lsc \left(\inf_{\eps>0} \cF_{\eps}(\cdot , \ouint)\right),
\end{equation}
where $\lsc(F)$ is the lower semicontinuous envelope of $F$; that is, for any $\bfu\in X$,
\begin{equation}
\lsc(F)(\bfu)=\sup \{ G (\bfu) : G \text{ is lower semicontinuous}, G\leq F\}.
\end{equation}

As a first step toward establishing the Sadowsky functional as the $\Gamma$-limit of the sequence $\{\cF_{\eps}(\cdot,\ouint)\}$, it is useful to compute the functional 
$\tcF$ defined for each $\bfu \in \Xp$ by
\begin{equation}\label{tcFdef}
\tcF(\bfu)=\inf_{\eps>0} \cF_{\eps}(\bfu,\ouint).
\end{equation}
\begin{lemma}\label{tcFlemma}
Let $X=\Xp$ or $\Xkp$. Given $\bfu\in X$, $\tcF$ defined in accord with \eqref{tcFdef} is given by
\begin{equation}\label{tcFdefb}
\tcF (\bfu)= \begin{cases}
	\cF(\bfu,\ouint), & \text{\rm if } \bfu \in X\cap \{\bfu : \eta' \in L^{\infty}(\ouint)\}, \\[10pt]
	+\infty, & \text{\rm otherwise.} %
 \end{cases}
\end{equation}
\end{lemma}
\begin{proof}
By the Monotone Convergence Theorem (see, for instance, Wheeden and Zygmund \cite{wheeden1977}), if $\cF_{\eps}(\bfu,\ouint)$ is finite for some $\eps>0$, then the limits involved in \eqref{tcFdef} may be exchanged to give
\begin{align}
\begin{split}
\inf_{\eps>0} \cF_{\eps}(\bfu, \ouint)={} & \lim_{\eps \searrow 0}  \int_{\ouint} \k^2 (1+\eta^2)^2 g(\eps \eta') \ds \\
 ={} & \int_{\ouint} \lim_{\eps \searrow 0} \k^2 (1+\eta^2)^2 g(\eps \eta') \ds \\ 
 ={} & \cF(\bfu).
\end{split}
\end{align}
Notice that $\tcF(\bfu,\ouint)\neq \cF(\bfu,\ouint)$ if and only if $\cF_{\eps}(\bfu,\ouint)=+\infty$ for all $\eps>0$ and $\cF(\bfu,\ouint)<+\infty$. 

Suppose that $\cF(\bfu,\ouint)<+\infty$.
Let $A=\{s\in \ouint : \eta'(s)=0\}$, $B_{\eps} = \{s \in \ouint : 0<|\eta'(s)|<2/\eps\}$, and $C_{\eps}=\ouint\setminus (A\cup B_{\eps})=\{s \in~\ouint : |\eta'(s)|\geq 2/\eps\}$.
Then
\begin{equation}
\cF_{\eps}(\bfu, \ouint)=\cF(\bfu,A)+\cF_{\eps}(\bfu,B_{\eps}) +\bfinf(C_{\eps}),
\end{equation}
where $\bfinf$ is the set function defined such that, given any measurable set $E$,
\begin{equation}
\bfinf(E)= \begin{cases}  +\infty & \text{ if } \mu(E)>0,\\ 0 & \text{ otherwise.}\end{cases}
\end{equation} 
However, $\cF_{\eps}(\bfu,\ouint)=+\infty$ for all $\bfu$ such that $\|\eta'\|_{\infty} =+\infty$ since, in that case $\bfinf(C_{\eps})=+\infty$ for all $\eps>0$.
On the other hand, if $\|\eta'\|_{\infty}<+\infty$, then for $\eps<2/\|\eta'\|_{\infty}$ it follows that $\bfinf(C_{\eps})=0$ and, by H\"{o}lder's inequality that
\begin{equation}
\cF_{\eps}(\bfu,B_{\eps})  \leq \cF(\bfu,B_{\eps}) \|g(\eps\eta'(\cdot))\|_{\infty}\leq g(\eps \|\eta'\|_{\infty}) \cF(\bfu,B_{\eps})<+\infty.
\end{equation}
Hence, $\tcF(\bfu,\ouint)\neq \cF(\bfu,\ouint)$ if and only if $\cF(\bfu,\ouint)<+\infty$ and $\|\eta'\|_{\infty}=+\infty$.   
\end{proof}

Lemma \ref{tcFlemma} and \eqref{glima} lead to the conclusion that
\begin{equation}
\Glim_{\eps\to 0^{+}} \cF_{\eps}(\cdot , \ouint)=\lsc (\tcF),
\end{equation}
where $\tcF$ is defined as in \eqref{tcFdefb}.

\subsection{Curves with curvature bounded from below}
Consider now the problem in which the space curve parameterized by $\bfu$ has (dimensionless) curvature $\k$ greater than some constant.
In particular, for $p>1$ and $\k_m>0$, take $\Xkp$ as defined in \eqref{fnspace}.
By the compact embedding $W^{3,p}(\ouint; \RR^3) \hookrightarrow C^{2,1-1/p}(\cluint; \RR^3)$, any function $\bfu\in \Xkp$ satisfies the pointwise constraints $|\bfu'(s)|=1$ and $|\bfu''(s)|\geq \k_m$ for almost every $s\in \ouint$. 
For information on the salient embedding results, see Adams and Fournier \cite{adams2003sobolev}.
\subsection{The Sadowsky functional is lower semicontinuous}
Notice that on the set $\Xkp$, the functional $\cF$ may be evaluated via
\begin{equation}
\cF(\bfu;\ouint)= \int_{\ouint} f(\bfu',\bfu'',\bfu''') \ds,
\end{equation}
where $f: \RR^3 \times \RR^3 \times \RR^3 \to \RR$ is defined by
\begin{equation}
f(\bfx,\bfy,\bfz)= \begin{cases}    |\bfy|_{\phantom{m}}^2 \left(1+\frac{[\bfx\cdot (\bfy\times \bfz)]^2}{|\bfy|^6} \right)^2, &\text{ for } |\bfy|\geq \k_m, \\ \phantom{|}\k_m^2\phantom{|} \left(1+\frac{[\bfx\cdot (\bfy\times \bfz)]^2}{\k_m^6} \right)^2, & \text{ for } |\bfy|<\k_m.
\end{cases}
\end{equation}
The map $f$ has the following properties: $f$ and the derivative $f_{\bfz}$ of $f$ with respect to its third argument are continuous; $f$ is convex in its third argument; $f$ is non-negative. A modification of Tonelli's semicontinuity theorem (see, for instance, Buttazzo, Giaquita and Hildebrandt \cite{Buttazzoetal1998}) implies that $\cF(\cdot,\ouint)$ is sequentially weakly lower semicontinuous in $W^{3,p}(\ouint;\RR^3)$ for all $p\geq 1$.
To present this modification of the theorem, it is useful to recall the following terminology.
\begin{definition}
A functional $\cG\colon \cB \to \RR$ is \emph{sequentially weakly lower semicontinuous} in a Banach space $\cB$, if for every $\bfx\in \cB$ and every sequence $\{\bfx_k\}\in \cB$ that converges weakly in $\cB$ to $\bfx$, the condition
\begin{equation}
\cG (\bfx) \leq \liminf_{k\to \infty} \cG(\bfx_k)
\end{equation}
is satisfied.
\end{definition}

\begin{theorem}[Modified Tonelli's semicontinuity theorem]
Let $I\subset \RR$ be a bounded open interval, and for $n\geq 1$ let $f\colon \RR^n \times \RR^n\times \RR^n \to \RR$ be a function with the following properties:
\begin{enumerate}
	\item $f$ and $f_{\bfc}$ are continuous in $(\bfa,\bfb,\bfc)$;
	\item $f$ is non-negative or bounded from below by an $L^1$ function;
	\item $f$ is convex in $\bfc$.
\end{enumerate}
The functional $\cF$ defined by
\begin{equation}
\cF (\bfu)=\int_{I} f(\bfu'(t),\bfu''(t),\bfu'''(t)) \dt
\end{equation}
is then sequentially weakly lower semicontinuous in $W^{3,p}(I; \RR^n)$ for all $p\geq 1$.
\end{theorem}
The proof follows that presented by Buttazzo, Giaquinta and Hildebrandt \cite{Buttazzoetal1998}.
\begin{proof}
Let a sequence $\{\bfu_k\}$ that converges weakly to $\bfu$ in $W^{3,p}(I; \RR^n)$ be given. 
Then it also converges to $\bfu$ weakly in $W^{3,1}(I; \RR^n)$ and strongly in $C^{1}(\bar{I}; \RR^n)$. 
In particular, $\{\bfu_k\}$ and $\{\bfu_k'\}$ converge uniformly on $\bar{I}$.
Passing to a subsequence, it is possible to assume that $\{\bfu_k''\}$ converges in $L^q (I;\RR^n)$ for every $q\geq 1$ and, hence, almost everywhere.

For any $\eps>0$, choose a compact subset $K \subset I$ such that, by Egorov's theorem,
$\bfu_k'' \to \bfu''$ uniformly on $K$ and, by Lusin's theorem (see, for example, Wheeden and Zygmund \cite{wheeden1977}), $\bfu$, $\bfu'$, $\bfu'',$ and $\bfu'''$ are continuous in $K$, and the measure $|I\setminus K|$ is sufficiently small to ensure that
if $\cF(\bfu)$ is finite, then
\begin{equation}\label{finiteineq}
\int_{K} f(\bfu',\bfu'',\bfu''')\dt \geq \int_{I} f(\bfu',\bfu'',\bfu''')\dt -\eps,
\end{equation}
and if, alternatively, $\cF(\bfu)$ is infinite, then
\begin{equation}\label{infiniteineq}
\int_{K} f(\bfu',\bfu'',\bfu''')\dt >\frac{1}{\eps}.
\end{equation}

In either case, since $f$ is convex in its third argument, then $\cF(\bfu_k)$ must obey
\begin{align}
	\begin{split}
\cF(\bfu_k)\geq {} & \int_{K} f(\bfu_k',\bfu_k'',\bfu_k''') \dt \\
\geq {} & \int_{K} f_{\bfc}(\bfu_k',\bfu_k'',\bfu''')\cdot (\bfu_k'''-\bfu''') +f(\bfu_k',\bfu_k'',\bfu''')\dt \\
={} & \int_{K} \left[ f_{\bfc}(\bfu_k',\bfu_k'',\bfu''')-f_{\bfc}(\bfu',\bfu'',\bfu''')\right]\cdot (\bfu_k'''-\bfu''')\dt+\int_{K} f_{\bfc}(\bfu',\bfu'',\bfu''')\cdot (\bfu_k'''-\bfu''')\dt
\\{} & \quad
+ \int_{K} f(\bfu_k',\bfu_k'',\bfu''')\dt.
	\end{split}
\end{align}
For the given choice of $K$, it follows that $f_{\bfc}(\bfu'(\cdot),\bfu''(\cdot),\bfu'''(\cdot))\in\nobreak L^{\infty}(K;\RR^3)$,
since $\bfu'$, $\bfu''$, and $\bfu'''$ are continuous on the compact set $K$ and $f_{\bfc}$ is assumed to be continuous on $\RR^n\times \RR^n\times \RR^n$.
Since $\{\bfu_k'''\}$ converges weakly to $\bfu'''$ in $L^{1}(K;\RR^3)$, it is possible to infer that
\begin{equation}
\int_{K} f_{\bfc}(\bfu',\bfu'',\bfu''')\cdot (\bfu_k'''-\bfu''')\dt \to 0 \text{ as } k\to \infty. 
\end{equation}
The weak convergence of the sequence $\{\bfu_k'''-\bfu'''\}$ to $\mathbf{0}$ in $L^{1}(I,\RR^3)$ implies that the sequence $\{\bfu_k'''-\bfu'''\}$ is equibounded in $L^{1}(I,\RR^3)$.
Moreover, $f_{\bfc}(\bfu_k',\bfu_k'',\bfu''')-f_{\bfc}(\bfu',\bfu'',\bfu''')$ converges uniformly to zero as $k\to \infty$.
Thus,
\begin{equation}
\int_{K} \left[ f_{\bfc}(\bfu_k',\bfu_k'',\bfu''')-f_{\bfc}(\bfu',\bfu'',\bfu''')\right]\cdot (\bfu_k'''-\bfu''')\dt\to 0 \text{ as } k \to \infty.
\end{equation}
Hence, appealing to the positivity of $f$ and \eqref{finiteineq}, for $\cF(\bfu)<+\infty$,
\begin{equation}
\liminf_{k\to \infty} \cF(\bfu_k) \geq \int_{K} f(\bfu',\bfu'',\bfu''') \dt \geq \cF(\bfu) -\eps.
\end{equation}
Similarly, by \eqref{infiniteineq}, if $\cF(\bfu)$ is infinite, then
\begin{equation}
\liminf_{k\to \infty} \cF(\bfu_k) \geq \int_{K} f(\bfu',\bfu'',\bfu''') \dt >\frac{1}{\eps}.
\end{equation}
Since $\eps>0$ is arbitrary, the conclusion follows.   
\end{proof}

A useful Lemma, which might be of independent interest, is next stated and proven.
\begin{lemma}\label{Fiscont}
For all $\k_m>0$ and $p\geq 4$, the Sadowsky functional $\cF(\,\cdot \,; \ouint)$ is continuous on $\Xkp$ with respect to strong convergence in $W^{3,p}(\ouint;\RR^3)$.
\end{lemma}
\begin{proof}
Let $\bfu\in \Xkp$ and a sequence $\{\bfu_n\}\subset \Xkp$ consistent with $\bfu_n\to \bfu$ as $n\to \infty$ be given.
Then
\begin{align}
	\begin{split}
	|\cF(\bfu,\ouint)-\cF(\bfu_n,\ouint) |
	=    {} & \left| \int_{\ouint} f(\bfu',\bfu'',\bfu''') \dt - \int_{\ouint} f(\bfu_n',\bfu_n'',\bfu_n''')\dt \right| \\
	\leq {} & \int_{\ouint} \left | f(\bfu',\bfu'',\bfu''')- f(\bfu_n',\bfu_n'',\bfu_n''') \right|\dt.		
	\end{split}
\end{align}
Further,
\begin{align}
\begin{split}
	|\cF(\bfu,\ouint)-\cF(\bfu_n,\ouint) |\leq {} &  \int_{\ouint} \left| |\bfu''|^2-|\bfu_n''|^2 \right| \dt +2 \int_{\ouint} \left| \frac{[\bfu'\cdot (\bfu''\times \bfu''')]^2}{|\bfu''|^4}- \frac{[\bfu_n'\cdot (\bfu_n''\times \bfu_n''')]^2}{|\bfu_n''|^4}\right| \dt \\
{} & +\int_{\ouint} \left | \frac{[\bfu'\cdot (\bfu''\times \bfu''')]^4}{|\bfu''|^{10}}- \frac{[\bfu_n'\cdot (\bfu_n''\times \bfu_n''')]^4}{|\bfu_n''|^{10}}\right |	\dt\\
={} & L+2M+N.
\end{split}
\end{align}
	Clearly, $L\to 0$ as $n\to \infty$ if $p\geq 2$. Next, consider the problem of estimating $M$ and $N$.
	
Since $|\bfu''|$ and $|\bfu_n''|$ are both greater than or equal to $\kappa_m$, it follows that
\begin{align}
	\begin{split}
		\k_m^{8} M \leq {} &   \int_{\ouint} \left| |\bfu_n''|^4 [\bfu'\cdot (\bfu''\times \bfu''')]^2 - |\bfu''|^4 [\bfu_n'\cdot (\bfu_n''\times \bfu_n''')]^2 \right| \dt\\
		  \leq {} &  \int_{\ouint} [\bfu'\cdot (\bfu''\times \bfu''')]^2 \left| |\bfu_n''|^4 -|\bfu''|^4 \right|\dt +\int_{\ouint} |\bfu''|^4 \left | [\bfu_n'\cdot (\bfu_n''\times \bfu''')]^2  - [\bfu_n'\cdot (\bfu_n''\times \bfu_n''')]^2 \right| \dt\\
			{} & +\int_{\ouint} |\bfu''|^4 \left | [\bfu_n'\cdot (\bfu_n''\times \bfu''')]^2  - [\bfu'\cdot (\bfu''\times \bfu''')]^2 \right| \dt\\
			={} & M_1+M_2+M_3.
	\end{split}
\end{align}
Notice that, since $\bfu \in C^{2,1-1/p} (\cluint; \RR^3)$, $\bfu'$ and $\bfu''$ obey $\bfu'\in L^{\infty}(\ouint;\RR^3)$ and $\bfu'' \in L^{\infty}(\ouint;\RR^3)$. 
By the general form of H\"{o}lder's inequality and the inequality $|\bfa\cdot (\bfb \times \bfc)|\leq |\bfa| |\bfb| |\bfc|$, $M_1$ is bounded above in accord with
\begin{equation}\label{Moneineq}
M_1 \leq \||\bfu_n''|^4 -|\bfu''|^4\|_{\infty} \| \bfu''\|^2_{\infty} \|\bfu'''\|_{2}^{2}.
\end{equation}
The compact embedding $W^{3,p}(\ouint;\RR^3) \hookrightarrow~C^{2,1-1/p} (\cluint;\RR^3)$ ensures that $|\bfu_n''|\to |\bfu''|$ uniformly. Hence, it is possible to infer that $M_1\to 0$ as $n\to \infty$.
Moreover, $M_2$ obeys
\begin{align}\label{Mtwoineq}
	\begin{split}
M_2 
\leq {} &\left \| \bfu'' \right \|^4_{\infty} \|\bfu_n' \cdot [\bfu_n'' \times (\bfu'''+\bfu_n''')] \|_{2} \|\bfu_n' \cdot [\bfu_n'' \times (\bfu'''-\bfu_n''')]\|_{2}\\
\leq {} &\left \| \bfu'' \right \|^4_{\infty} \left \| \bfu_n'' \right \|^2_{\infty} \left( \|\bfu'''\|_{2} +\|\bfu_n''' \|_{2}\right) \|\bfu'''-\bfu_n'''\|_{2}.
	\end{split}
\end{align}
Since $\bfu''_n\to \bfu''$ uniformly on $\ouint$ and $\bfu'''_n\to \bfu'''$ in $L^2(\ouint;\RR^3)$, it follows that $M_2\to 0$ as $n\to \infty$.
On permuting the triple products, it follows that
\begin{align}\label{Mthrineq}
	\begin{split}
M_3
\leq {} &\left \| \bfu'' \right \|^4_{\infty} \| \bfu''' \cdot (\bfu_n' \times \bfu''_n+\bfu'\times \bfu'')\|_{2} \|\bfu'''\cdot (\bfu_n'\times \bfu_n''-\bfu'\times \bfu'')\|_{2} \\
\leq {} &\left \| \bfu'' \right \|^4_{\infty} \left \| \bfu''' \right \|_{2}^2  \|\bfu_n' \times \bfu''_n+\bfu'\times \bfu''\|_{\infty} \|\bfu_n'\times \bfu_n''-\bfu'\times \bfu''\|_{\infty}.
\end{split}
\end{align}
Thus, by the continuity of the cross product and the uniform convergence of $\{\bfu_n'\}$ and $\{\bfu_n''\}$ to $\bfu'$ and $\bfu''$, respectively, $M_3\to 0$ as $n\to \infty$.

Similarly,
\begin{align}
	\begin{split}
		\k_m^{20}N 
		\leq {}  & \int_{\ouint} [\bfu'\cdot (\bfu''\times \bfu''')]^{4} \left | |\bfu''|^{10} -|\bfu_n''|^{10} \right| \dt 
		+\int_{\ouint} |\bfu''|^{10} \left| [\bfu_n'\cdot(\bfu_n''\times\bfu''')]^4 -[\bfu_n'\cdot(\bfu_{n}''\times \bfu_n''')]^{4}\right| \dt\\
		{} & \int_{\ouint} |\bfu''|^{10} \left| [\bfu_n'\cdot (\bfu_n''\times \bfu''')]^4-[\bfu'\cdot (\bfu''\times \bfu''')]^4\right| \dt\\
		={} & N_1+N_2+N_3.
	\end{split}
\end{align}
Calculations analogous to those leading to the bound \eqref{Moneineq} satisfied by $M_1$ yield
\begin{equation}
N_1 \leq \| |\bfu''|^{10} -|\bfu_n''|^{10} \|_{\infty} \| \bfu'' \|^{4}_{\infty} \|\bfu'''\|_{4}^{4}.
\end{equation}
Factoring the integrand of $N_2$ and repeatedly applying H\"{o}lder's inequality leads to
\begin{align}
	\begin{split}
N_2 
\leq {} & \|\bfu''\|_{\infty}^{10} \|\bfu''_n\|_{\infty}^{4} \left(\|\bfu'''\|_{4}^2+\|\bfu_n'''\|_{4}^{2}\right) \|\bfu'''+\bfu_n'''\|_{4} \|\bfu'''-\bfu_n'''\|_{4}\\
\leq {} &  \|\bfu''\|_{\infty}^{10} \|\bfu''_n\|_{\infty}^{4} \left(\|\bfu'''\|_{4}+\|\bfu_n'''\|_{4}\right)^3\|\bfu'''-\bfu_n'''\|_{4},
	\end{split}
\end{align}
where, in the second line, Minkowski's theorem and the inequality $(a+b)^3 \geq (a^2+b^2)(a+b)$, for $a,b\geq 0$, have been used.
Similarly,
\begin{align}
	\begin{split}
		N_3 \leq {} & \|\bfu''\|_{\infty}^{10} \| [\bfu'''\cdot (\bfu_n'\times \bfu_n'')]^2+[\bfu'''\cdot (\bfu'\times \bfu'')]^2 \|_{2} \|\bfu'''\cdot (\bfu_n'\times \bfu_n''+\bfu'\times \bfu'')\|_{4}\\
		{} & \quad\|\bfu'''\cdot (\bfu_n'\times \bfu_n''-\bfu'\times \bfu'')\|_{4}\\
		\leq {} & \|\bfu''\|_{\infty}^{10} \|\bfu'''\|_{4}^{4}  \left(\| \bfu_n'\times \bfu_n''\|_{\infty}+\|\bfu'\times \bfu''\|_{\infty} \right)^3 \|\bfu_n'\times \bfu_n''-\bfu'\times \bfu'' \|_{\infty}.
	\end{split}
\end{align}
On taking $p\geq 4$, it is clear that $N_1,N_2,N_3\to 0$ as $n\to \infty$.

This proves that $\lim_{n\to \infty} \cF(\bfu_n,\ouint)=~\cF(\bfu,\ouint)$.\phantom{.}
\hfill\hfill  
\end{proof}

\begin{theorem}
If $p\geq 4$, any weakly sequentially lower semicontinuous function $\cG\colon \Xkp \to \RR$ that obeys $\cG(\bfu)\leq \cF (\bfu,\ouint)$ on the dense subset $\Yk$, also satisfies the inequality 
\begin{equation}
\cG(\bfu) \leq \cF(\bfu,\ouint)
\end{equation}
for all $\bfu \in \Xkp$. It follows that for $\bfu\in \Xkp$
\begin{equation}
\Glim_{\eps\searrow 0} \cF_{\eps}(\bfu, \ouint) = \cF(\bfu,\ouint)
\end{equation}
with respect to the weak topology on $W^{3,p}(\ouint;\RR^3)$ for $p\geq 4$.
\end{theorem}

\begin{proof}
Let $\bfu\in \Xkp$ be given. 
By hypothesis, $\cG$ is weakly lower semicontinuous, and, therefore, if there exists a sequence $\{\bfu_n\}\subset \Yk$ such that $\bfu_n\rightharpoonup \bfu$ in $W^{3,p}(\ouint; \RR^3)$, then $\cG$ is bounded above in accord with
\begin{equation}
\cG(\bfu) \leq \liminf_{n\to \infty} \cG(\bfu_n)\leq \liminf_{n\to\infty} \cF(\bfu_n,\ouint).
\end{equation} 
Clearly, any curve $\bfu\in \Yk$ satisfies $\|\eta'\|_{\infty}<+\infty$. 
Therefore, it suffices to find a sequence $\{\bfu_n\}\subset \Yk$ such that $\bfu_n \rightharpoonup \bfu$ in $W^{3,p}(\ouint; \RR^3)$ and
the limit
 \begin{equation*}\lim_{n\to \infty} \cF(\bfu_n,\ouint)=\cF(\bfu,\ouint)\end{equation*}
holds.
By definition of $\Xkp$, it is possible to choose a sequence $\{\bfu_n\} \subset Y$ such that $\bfu_n\to \bfu$ strongly (and, therefore weakly, as well) in $W^{3,p}(\ouint;\RR^3)$.
By Lemma \ref{Fiscont}, $\cF(\bfu,\ouint)=\lim_{n\to \infty} \cF(\bfu_n,\ouint)$, and the conclusion follows.   
\end{proof}
\begin{theorem}
If $p\geq 4$, any lower semicontinuous function $\cG\colon \Xkp \to \RR$ that obeys $\cG(\bfu)\leq \cF (\bfu,\ouint)$ on the dense subset $\Yk$, also satisfies the inequality 
\begin{equation}
\cG(\bfu) \leq \cF(\bfu,\ouint)
\end{equation}
for all $\bfu \in \Xkp$. It follows that for $\bfu\in \Xkp$
\begin{equation}
\Glim_{\eps\searrow 0} \cF_{\eps}(\bfu, \ouint) = \cF(\bfu,\ouint)
\end{equation}
with respect to the strong topology on $W^{3,p}(\ouint;\RR^3)$ for $p\geq 4$.
\end{theorem}
\begin{proof}
Let $\bfu\in \Xkp$ be given. 
As before, it suffices to find a sequence $\{\bfu_n\}\subset \Yk$ such that $\bfu_n \to \bfu$ in $W^{3,p}(\ouint; \RR^3)$ and
the limit
 \begin{equation*}\lim_{n\to \infty} \cF(\bfu_n,\ouint)=\cF(\bfu,\ouint)\end{equation*}
holds.
By Lemma \ref{Fiscont}, $\cF(\bfu,\ouint)=\lim_{n\to \infty} \cF(\bfu_n,\ouint)$, and the conclusion follows.   
\end{proof}

The theorem have the following important corollary.
\begin{corollary}
On the set $X^{+,p}=\bigcup\limits_{\k_m>0} \Xkp$, the Sadowsky functional is the $\Gamma$-limit of the elastic energy $\cF_{\eps}(\cdot,\ouint)$ with respect to the weak topology and with respect to the strong topology on $W^{3,p}(\ouint;\RR^3)$.
\end{corollary}
\begin{proof}
Let $\bfu\in X^{+,p}$. Then $\bfu\in \Xkp$ for some $\km>0$.
That the limsup condition \eqref{limsupcond} is satisfied follows from the $\Gamma$ convergence of $\cF_{\eps}(\cdot,I)$ to $\cF(\cdot,I)$ in $\Xkp$.

Let $\{\bfu_j\}$ be a sequence in $X^{+,p}$ such that $\bfu_j\to \bfu$ either strongly or weakly.
Then $\bfu \in \Xkp$ for some $\km>0$ and for either modes of convergence, $\bfu_j''\to \bfu$ uniformly on $\cluint$.
Hence, $\bfu_j\in X^{\km/2,p}$ for $j\geq J$ for some sufficiently large $J>0$.
Thus \eqref{liminfcond} holds since $\cF_{\eps_j} (\cdot ,I)$ $\Gamma$-converges to $\cF(\cdot,I)$ in $X^{\km/2,p}$.
\end{proof}

The set $X^{+,p}$ consists of those arclength parameterized curves in $W^{3,p}(I;\RR^3)$ well-approximated by a sequence of smooth curves for which the infima of their curvatures is bounded away from zero. The space $X^{+,p}$ is the natural space of curves with nonvanishing curvature described in the Introduction.
\section{Lower semicontinuity at curves with isolated inflection points}
A limitation of this analysis is that nonorientable ribbons do not have centerlines in the space $X^{+,p}$, since all such centerlines have strictly positive curvature.
Since the Sadowsky functional was proposed in the context of the shape of a M\"{o}bius band, it is of interest to determine whether the $\Gamma$-convergence result extends to a space of centerlines which allows, at least, for isolated inflection points. 
The following lemma provides a step in this direction.
\begin{lemma}
The Sadowsky functional $\cF$ is sequentially weakly lower semicontinuous in $W^{3,p}(I;\RR^3)$ at functions $\bfu\in \Xp$ such that 
$\{t\in \cluint: \bfu''(t) =0\}$ has measure zero for $p>1$.
\end{lemma}
\begin{proof}
Given an element $\bfu\in \Xp$ for $p>1$, define $Z=\{s\in \cluint: \bfu''(s) =0\}$ and let $\{I_j\}_{j\in \cJ}$ be the family of nonoverlapping open intervals $I_j=(a_j,b_j)$ on which $\bfu''(s)\neq 0$. 
Since there are at most countably many such intervals, take $\cJ \subset \NN$.
Since $Z$ has measure zero, it is clear that 
\begin{equation}
\cF(\bfw,I)=\cF\bigg(\bfw,\bigcup_{j\in \cJ} I_j\bigg),
\end{equation} 
for any $\bfw\in \Xp$.

Let $\{\bfu_n\}$ be a sequence in $\Xp$ such that $\bfu_n\wto \bfu$ weakly in $W^{3,p}(\ouint;\RR^3)$, and let $j\in \cJ$ be given.
Choose $K_j\in \NN$ so large that $K_j>2/(b_j-a_j)$. 
For $k\geq K_j$, then define the intervals $I_{j,k}$ by
\begin{equation}
I_{j,k}=\left(a_j+\frac{1}{k},b_j-\frac{1}{k}\right),
\end{equation}
with the remainder denoted by $R_{j,k}=I_j\setminus I_{j,k}$.
By the continuity of $\bfu''$, the curvature $|\bfu''|$ is bounded below on each interval $I_{j,k}$  such that
\begin{equation}
m_{j,k}=\inf_{t\in I_{j,k}} |\bfu''(t)| >~0.
\end{equation}
By the uniform convergence $\bfu_n''\to \bfu''$ on $\cluint$, granted that $N_j$ is sufficiently large it can be concluded that $|\bfu_n''(t)|>\frac{m_{j,k}}{2}$ for all $t\in I_{j,k}$ and $n\geq N_j$.
 For any $\eps>0$ it is feasible to choose $K_j>0$ such that if $k_j>K_j$, then
$\cF(\bfu, I_{j,k_j}) >\cF(\bfu, I_j)-\eps 2^{-j}$ if $\cF(\bfu,I_j)$ is finite and $\cF(\bfu, I_{j,k_j}) >1/\eps$ otherwise.

Suppose that $\cF(\bfu,I_{j})$ is finite for all $j\in \cJ$. 
By Fatou's lemma applied to the counting measure, the countable additivity of $\cF(\bfw;\cdot)$ as a set function, and the positivity of $\cF(\cdot,A)$ for any measurable set $A\subset I$
\begin{align*}
\liminf_{n\to \infty} \cF(\bfu_n,\ouint) = {} & \liminf_{n\to \infty} \sum_{j\in \cJ} \cF(\bfu_n,I_j) \\
 \geq {} & \sum_{j\in \cJ} \liminf_{n\to \infty} \cF(\bfu_n,I_{j,k_j}).
\end{align*}
Since $\cF(\cdot, I_{j,k_j})$ is sequentially weakly lower semicontinuous at functions such that $|\bfu''|\geq m$ for some $m>0$, it follows that
\begin{align*}
\liminf_{n\to \infty} \cF(\bfu_n,\ouint) \geq {} &  \sum_{j\in \cJ} \cF(\bfu,I_{j,k_j})\\
= {} &  \cF\bigg(\bfu,\bigcup_{j\in \cJ} I_{j}\bigg)-\eps,
\end{align*}
wherein countable additivity of $\cF(\bfu, \cdot)$ as a set function is used.

If there is at least one interval $I_{J}$ for which $\cF(\bfu,I_J)$ is infinite, then $\cF(\bfu,I)$ is infinite. 
Moreover, it follows that
\begin{align*}
\begin{split}
	\liminf_{n\to \infty} \cF (\bfu_n,I) \geq {} & \liminf_{n\to \infty} \cF (\bfu_n,I_{J,k_J})\\
	\geq {} & \frac{1}{\eps}.
\end{split}	
\end{align*}

Since $\eps>0$ is arbitrary, the conclusion holds.  
\end{proof}

\section{Discussion}
Some final remarks are in order. 
It has been shown that the $\Gamma$-convergence in a space $X^{+,p}$ of ribbons with centerlines having nonvanishing curvature with respect 
to weak and strong convergence of the centerlines in $W^{3,p}(I;\RR^3)$.
The elastic energy $\cF_{\eps}$ of such ribbons may be written as a single integral along their centerline 
depending parametrically on the aspect ratio $\eps$ of the ribbon.
For a fixed centerline $\bfu$, the elastic energy $\cF_{\eps}(\bfu,I)$ is monotonically increasing in the aspect ratio, and the $\Gamma$-limit result follows upon showing that the lower semicontinuous envelope of the point-wise limit $\tilde{\cF}$ is the Sadowsky functional in $X^{+,p}$ for $p\geq 4$ with respect to weak convergence in 
$W^{3,p}(I;\RR^3)$.

It has also been shown that the Sadowsky functional is 
sequentially weakly lower semicontinuous at functions $\bfu$ in $\Xp$ such that the set of inflection points has measure zero. 
This is a necessary but not sufficient condition for the Sadowsky functional to be the $\Gamma$-limit in a space containing such functions since $\Gamma$-limits are lower semicontinuous. 
To establish that $\tcF$ the $\Gamma$-limit of $\cF_{\eps}$ with respect to the weak (strong) topology on $W^{3,p}(I;\RR^3)$, 
it remains to construct a sequence of functions $\{\bfu_n\}$ in $\Xp$ with $\eta_n$ satisfying $\eta'_n\in L^{\infty}(\ouint)$, 
$\bfu_n\to \bfu$ in weakly (strongly, respectively) $W^{3,p}(\ouint;\RR^3)$, and $\lim_{n\to \infty} \cF(\bfu_n,I)=\cF(\bfu,I)$.
Such a construction would firmly establish Sadowsky's functional as the variational limit in a space including centerlines of developables that form M\"{o}bius bands---the context in which that functional was originally derived.

Finally, from the perspective of direct methods in the calculus of variations, it is of interest to prove the following natural conjecture.
\begin{conjecture}
For each $\km>0$ and $p\geq 4$, smooth functions are strongly dense in the space
\begin{equation}
\Xikp = \{ \bfu \in W^{3,p} (I;\RR^3) : D^{[1]}\bfu(0)=\bfA, D^{[1]}\bfu(1)=\bfB, |\bfu'(s)|=1 \text{ and } |\bfu''(s)|\geq \km \text{ for a.e. } s\in I\}.
\end{equation}
In particular, $\Xkp=\Xikp$, and consequently $\Xkp$ is sequentially weakly closed.
\end{conjecture}
Although this direction is not pursued here, it is also of interest to show that there exist a minima of the functionals $\cF_{\eps}(\cdot , I)$ and $\cF(\cdot, I)$ 
on the set $\Xkp$ for each $\km>0$, $\eps>0$, and $p\geq 4$, 
and that the minima of $\cF_{\eps}(\cdot, I)$ converge to the minimum of $\cF(\cdot,I)$. 
Establishing the above conjecture is likely to be an important step toward esblishing the existence of minimizers in $\Xkp$.


\bibliographystyle{alpha}      
\bibliography{Gamma_bib}   

\begin{thebibliography}{SvdH07}

\bibitem[AF03]{adams2003sobolev}
R.~A. Adams and J.~J.~F. Fournier.
\newblock {\em Sobolev Spaces}.
\newblock Academic Press, 2003.

\bibitem[BGH98]{Buttazzoetal1998}
G.~Buttazzo, M.~Giaquinta, and S.~Hildebrandt.
\newblock {\em One-dimensional Variational Problems}.
\newblock Oxford University Press, 1998.

\bibitem[Bra02]{Braides}
A.~Braides.
\newblock {\em $\Gamma$-convergence for Beginners}.
\newblock Oxford University Press, 2002.

\bibitem[Ger21]{Germain1821}
S.~Germain.
\newblock {\em Recherches sur la th{\'e}orie des surfaces {\'e}lastiques}.
\newblock Huzard-Courcier, 1821.

\bibitem[Gra35]{Graustein1935}
W.~C. Graustein.
\newblock {\em Differential Geometry}.
\newblock MacMillan, 1935.

\bibitem[Poi12]{Poisson1812}
S.~D. Poisson.
\newblock M\'{e}moire sur les surfaces \'{e}lastiques.
\newblock {\em M\'{e}m. des sci. math. et phys.}, pages 167--226, 1812.

\bibitem[RR96]{Randrup1996}
T.~Randrup and P.~R{\o}gen.
\newblock Sides of the {M}{\"o}bius strip.
\newblock {\em Arch. der Math.}, 66(6):511--521, 1996.

\bibitem[Sad30]{Sadowsky1930}
M.~Sadowsky.
\newblock Ein elementarer {B}eweis f{\"u}r die {E}xistenz eines abwickelbaren
  {M}{\"o}biusschen {B}andes und {Z}ur{\"u}ckf{\"u}rung des geometrischen
  {P}roblems auf ein {V}ariationsproblem.
\newblock {\em Sitzber. Preussischen Akad. der Wiss. Philos.-hist. Kl.},
  22:412--415, 1930.

\bibitem[Str61]{Struik1961}
D.~J. Struik.
\newblock {\em Lectures on Classical Differential Geometry}.
\newblock Dover, 1961.

\bibitem[SvdH07]{SvdH2007}
E.~L. Starostin and G.~H.~M. van~der Heijden.
\newblock The equilibrium shape of an elastic developable {M}\"{o}bius strip.
\newblock {\em Proc. in Appl. Math. and Mech.}, 7(1):2020115--2020116, 2007.

\bibitem[Wun62]{Wunderlich1962}
W.~Wunderlich.
\newblock {\"U}ber ein abwickelbares {M}{\"o}biusband.
\newblock {\em Mon. f{\"u}r Math.}, 66(3):276--289, 1962.

\bibitem[WZ77]{wheeden1977}
R.~L. Wheeden and A.~Zygmund.
\newblock {\em Measure and Integral}.
\newblock Marcel Dekker, 1977.

\bibitem[Yon12]{Yongthesis}
E.~H. Yong.
\newblock {\em Problems at the Nexus of Geometry and Soft Matter: Rings,
  Ribbons and Shells}.
\newblock PhD thesis, Harvard University, 2012.

\end{thebibliography}


\end{document}